\documentclass{amsart}
\usepackage{amsmath, amsthm, amssymb} %, latexsym, graphics}

\newtheorem{theorem}{Theorem}[section]
\newtheorem{lemma}[theorem]{Lemma}
\newtheorem{proposition}[theorem]{Proposition}
\newtheorem{ex}{Exercise}[section]
\newtheorem{examp}[theorem]{Example}
\newtheorem{corollary}[theorem]{Corollary}
\newtheorem{remar}[theorem]{Remark}

\newenvironment{remark}{\begin{remar}\rm}{\end{remar}}

\newcommand{\diams}{\unskip\nobreak\hfil\penalty50%
\hskip1em\hbox{}\nobreak\hfil%
$\diamondsuit$\parfillskip=0pt\finalhyphendemerits=0}

\newcommand{\bfind}[1]{\index{#1}{\bf #1}}

\newcommand{\n}{\par\noindent}

\newcommand{\sn}{\par\smallskip\noindent}
\newcommand{\mn}{\par\medskip\noindent}
\newcommand{\bn}{\par\bigskip\noindent}
\newcommand{\pars}{\par\smallskip}
\newcommand{\parm}{\par\medskip}
\newcommand{\parb}{\par\bigskip}
\newcommand{\isom}{\simeq}

\newcommand{\ovl}[1]{\overline{#1}}

\newcommand{\Th}{\mbox{\rm Th}\,}
\newcommand{\chara}{\mbox{\rm char}\,}

\newcommand{\subsetuneq}{\mathrel{\raisebox{.8ex}{\footnotesize%
$\displaystyle\mathop{\subset}_{\not=}$}}}

\newcommand{\cal}{\mathcal}
\font\tenlv=msbm10  %scaled 1200
\font\sevenlv=msbm7 %scaled 1200
\font\fivelv=msbm5  %scaled 1200
\def\lv #1{{\mathchoice{{\hbox{\tenlv #1}}}{{\hbox{\tenlv #1}}}
{{\hbox{\sevenlv #1}}}{{\hbox{\fivelv #1}}}}}

\newcommand{\N}{\lv N}
\newcommand{\Q}{\lv Q}

\newcommand{\R}{\lv R}
\newcommand{\Z}{\lv Z}
\newcommand{\F}{\lv F}
\newcommand{\Fp}{\F_p}
\begin{document}
\title{Notes on extremal and tame valued fields\\
%\small\it --- Preliminary version ---
}
\author{Sylvy Anscombe and Franz-Viktor Kuhlmann}
\address{Jeremiah Horrocks Institute,
Leighton Building Le7,
University of Central Lancashire,
Preston,
PR1 2HE,
United Kingdom}
\email{sanscombe@uclan.ac.uk}
\address{Institute of Mathematics,
University of Silesia,
ul.~Bankowa 14,
40-007 Katowice,
Poland}
\email{fvk@math.us.edu.pl}
%\thanks{This work was partially supported by a Canadian NSERC grant }
\date{February 9, 2016}
%\subjclass[2000]{Primary 12J20; Secondary 12J10}
\begin{abstract}\noindent
%{\footnotesize\rm  }
We extend the characterization of extremal valued fields given in
\cite{[AKP]} to the missing case of valued fields of mixed
characteristic with perfect residue field. This leads to a complete
characterization of the tame valued fields that are extremal. The key to
the proof is a model theoretic result about tame valued fields in mixed
characteristic. Further, we prove that in an extremal valued field of
finite $p$-degree, the images of all additive polynomials have the
optimal approximation property. This fact can be used to improve the
axiom system that is suggested in \cite{[K1]} for the elementary theory
of Laurent series fields over finite fields. Finally we give examples
that demonstrate the problems we are facing when we try to characterize
the extremal valued fields with imperfect residue fields. To this end,
we describe several ways of constructing extremal valued fields; in
particular, we show that in every $\aleph_1$ saturated valued field
the valuation is a composition of extremal valuations of rank 1.

\end{abstract}
\maketitle

%
%???????????????????????????????????????????????????????????????????????
%
\section{Introduction}
A valued field $(K,v)$ with valuation ring ${\cal O}$ and value group
$vK$ is called \bfind{extremal} if for every multi-variable
polynomial $f(X_1, \dots , X_n)$ over $K$ the set
\[
\{v(f(a_1,\dots,a_n))\mid a_1,\dots, a_n \in {\cal O}\}
\subseteq vK\cup\{\infty\}
\]
has a maximal element. For the history of this notion, see \cite{[AKP]}.
In that paper, extremal fields were characterised in several special
cases, but some cases remained open. In the present paper we answer the
question stated after Theorem~1.2 of \cite{[AKP]} to the positive,
thereby removing the condition of equal characteristic from the theorem.
The most comprehensive version of the theorem now reads:

\begin{theorem}                                       \label{gen}
Let $(K,v)$ be a nontrivially valued field. If $(K,v)$ is extremal, then
it is algebraically complete and
\begin{itemize}
  \item[(i)] $vK$ is a $\mathbb{Z}$-group, or
  \item[(ii)] $vK$ is divisible and $Kv$ is large.
 \end{itemize}
Conversely, if $(K,v)$ is algebraically complete and
\begin{itemize}
  \item[(i)] $vK\isom\Z$, or $vK$ is a $\mathbb{Z}$-group and $\chara Kv=0$, or
  \item[(ii)] $vK$ is divisible and $Kv$ is large and perfect,
 \end{itemize}
then $(K,v)$ is extremal.
\end{theorem}
Note that a valued field $(K,v)$ is called \bfind{algebraically
complete} if every finite algebraic extension $(L,v)$ satisfies
\begin{equation}                            \label{dle}
[L:K]\>=\>(vL:vK)[Lv:Kv]\>,
\end{equation}
where $Lv$, $Kv$ denote the respective residue fields. Every
algebraically complete valued field $(K,v)$ is \bfind{henselian}, i.e.,
$v$ admits a unique extension to its algebraic closure $\tilde{K}$
(which we will again denote by $v$). Also, every
algebraically complete valued field $(K,v)$ is \bfind{algebraically
maximal}, that is, does not admit proper algebraic immediate extensions
$(L,v)$ (\bfind{immediate} means that $vL=vK$ and $Lv=Kv$). For later use let us mention
that a valued field is called \bfind{maximal} if it does not admit proper immediate extensions at all.

Further, $(K,v)$ is a \bfind{tame field} if it is henselian, perfect, and
$\tilde{K}$ is equal to the ramification field of the extension
$(\tilde{K}|K,v)$. All tame fields are algebraically complete (cf.\
\cite[Lemma~3.1]{[K4]}).

A field $K$ is \bfind{large} if every smooth curve over $K$ which has a
$K$-rational point, has infinitely many such points. For more
information about large fields, see \cite{[P]}, \cite{[K2a]} and
\cite{[AKP]}.

\parm
In \cite{[AKP]} it was proved that an algebraically complete valued field $(K,v)$ with
divisible value group and large perfect residue field is extremal if $\chara K =
\chara Kv$ (the \bfind{equal characteristic} case). To this end, we used
the Ax--Kochen--Ershov Principle
\begin{equation}                            \label{AKEequiv}
vK\equiv vL\>\wedge\> Kv\equiv Lv \;\;\;\Longrightarrow\;\;\;
(K,v)\equiv \mbox{$(L,v)$}
\end{equation}
which holds for all tame valued fields of equal characteristic (see
\cite[Theorem~1.4]{[K4]}). We were not able to cover the \bfind{mixed characteristic} case
$\chara K \ne\chara Kv$ because the principle was not known for this case. In fact,
we will show below (Theorem~\ref{tmcne}) that it is false. However, we
can do with lesser tools that {\it are} known. After all, at least the
corresponding Ax--Kochen--Ershov Principle for elementary extensions has
been proved in \cite{[K4]}:

\begin{theorem}                             \label{AKE}
If $(L|K,v)$ is an extension of tame fields such that $vK\prec vL$ and
$Kv\prec Lv$, then $(K,v)\prec (L,v)$.
\end{theorem}

This theorem enables us to prove:

\begin{theorem}                             \label{triangle}
Take a nontrivially valued tame field $(K,v)$ and two ordered abelian
groups $\Gamma$ and $\Delta$ such that $\Gamma\prec vK$ and $\Gamma
\prec\Delta$. Then there exist two tame fields $(K',v)$ and $(L,v)$ with
$vK'=\Gamma$, $vL=\Delta$, $Kv=K'v= Lv$, $(K',v)\prec (K,v)$ and
$(K',v)\prec (L,v)$. In particular, $(K,v)\equiv (L,v)$.
\end{theorem}

If $vK$ is nontrivial and divisible and $\Delta$ is any nontrivial
divisible ordered abelian group, then we can take $\Gamma=\Q$ to obtain
that $\Gamma\prec vK$ and $\Gamma\prec\Delta$ since the elementary
class of nontrivial divisible ordered abelian groups is model complete.
Thus, Theorem~\ref{triangle} yields the following result:

\begin{corollary}                           \label{tfeqtf}
If $(K,v)$ is a nontrivially valued tame field with divisible value
group and $\Delta$ is any nontrivial divisible ordered abelian group,
then there is a tame field $(L,v)\equiv (K,v)$ with $vL=\Delta$ and
$Lv=Kv$.
\end{corollary}

It is easy to see that (\ref{AKEequiv}) cannot hold in this generality in
the mixed characteristic case. One can construct two algebraic
extensions $(L,v)$ and $(L',v')$ of $(\Q,v_p)$, where $v_p$ is the
$p$-adic valuation on $\Q$, both having residue field $\F_p$, such that:
\sn
1) $L$ does not contain $\sqrt{p}$ and $vL$ is the $p$-divisible hull of
$(v_pp)\Z$,
\n
2) $L'$ contains $\sqrt{p}$ and $v'L'$ is the $p$-divisible hull of
$(v_p\sqrt{p})\Z=\frac{1}{2}(v_pp)\Z$.
\sn
Then $vL\isom v'L'$ and hence $vL\equiv v'L'$, but $(L,v)\not\equiv (L',v')$.

\pars
One could hope, however, that this problem vanishes when one strengthens
the conditions by asking that $vL$ and $v'L'$ are equivalent over
$v_p\Q$ (and $Lv$ and $L'v'$ are equivalent over $\Q v_p$). But the
problem remains:

\begin{theorem}                             \label{tmcne}
Take any prime $p$. Then the there exist valued field extensions
\[
(\Q,v_p)\subset (L_0,v)\subset (L_1,v)\subset (L_2,v)\;\mbox{ and }\;
(\Q,v_p)\subset (F_0,v)\subset (F_1,v)\subset (F_2,v)
\]
such that the following assertions hold:
\sn
a) The fields $(L_0,v)$ and $(F_0,v)$ are extensions of degree $p(p-1)$ of the henselization
of $\Q$ under the $p$-adic valuation and extremal with $L_0v=\F_p=F_0v$ and $vL_0=vF_0=\frac{1}{p(p-1)}
(v_p p)\Z$, but $(L_0,v)\not\equiv (F_0,v)$.
\sn
b) The fields $(L_1,v)$ and $(F_1,v)$ are algebraic over $\Q$ and tame with $L_1 v=F_1 v=\F_p$ and
$vL_1=vF_1$ equal to the $p$-divisible hull of $\frac{1}{p-1} (v_p p)\Z$, but
$(L_1,v)\not\equiv (F_1,v)$.
\sn
c) The fields $(L_2,v)$ and $(F_2,v)$ are tame and extremal, with perfect
residue fields $L_2v=F_2v$ and $vL_2=vF_2=\Q$, but $(L_2,v)\not\equiv (F_2,v)$.
\end{theorem}

\begin{corollary}
The Ax--Kochen--Ershov Principle (\ref{AKEequiv}) fails for extremal fields
with value group isomorphic to $\Z$ in mixed characteristic. It also fails for tame extremal fields
with value group isomorphic to $\Q$ and perfect residue field in mixed characteristic.
\end{corollary}
\sn
{\bf Open problem:} Can the situation be improved by adding the Macintyre power predicates to the language?

\pars
Note that $(L,v)\equiv (L',v')$ if and only if they are equivalent over
$(\Q,v_p)$, and this in turn holds if and only if we have the
equivalence
\[
(L,v)_\delta\equiv (L',v')_\delta\quad\mbox{over}\quad (\Q,v_p)_\delta
\]
of their amc structures of level $\delta$, for all $\delta\in (v_pp)\Z$
(see \cite[Corollary 2.4]{[K0]}). But this fact is of little use for
the proof of Corollary~\ref{tfeqtf} since it is by no means clear how to
construct an extension of $(\Q,v_p)$ whose amc structures of level
$\delta$ are equivalent to those of $(K,v)$.

\parm
The improvement in Theorem~\ref{gen} yields a corresponding improvement
of Proposition~5.3 from \cite{[AKP]}. Note that when we speak of a
composition $v=w\circ\ovl{w}$ of valuations, we do not mean a composition
as functions, but in fact refer to the composition of their associated places. That is, if $Q$ and $\bar{Q}$ are
the places associated with $w$ and $\bar{w}$, then their composition (with the obvious additional rules for
$\infty$) is the place associated with $v$.

\begin{proposition}                         \label{compimpr}
Take a valued field $(K,v)$ with perfect residue field. Assume that $v$
is the composition of two nontrivial valuations: $v=w\circ\ovl{w}$. Then
$(K,v)$ is extremal with divisible value group if and only if the same
holds for $(K,w)$ and $(Kw,\ovl{w})$.
\end{proposition}

We may say that a property $P$ of valuations is {\bf compatible with composition}
if $P(v)\Leftrightarrow P(w)\wedge P(\bar{w})$ for each composition $v=w\circ \bar{w}$.
Examples of such properties are ``henselian'', ``maximal'', ``algebraically complete'',
``divisible value group''. The latter two will be used in the proof of the proposition,
given in Section~\ref{sect2}. The proposition in fact yields that also the property
``extremal with divisible value group and perfect residue field'' is compatible with
composition (since if $(Kw,\bar{w})$ has this property, then in particular it is perfect).

It should be noted that the condition on the value groups cannot be
dropped without a suitable replacement, even when all residue fields
have characteristic $0$. Indeed, if the value group of $(K,w)$ is a
$\Z$-group and $\ovl{w}$ is nontrivial, then the value group of $(K,v)$
is neither divisible nor a $\Z$-group and $(K,v)$ cannot be extremal.
\pars
Let us state two\n
{\bf Open problems:}\sn
{\bf 1)} \ If $v=w\circ\bar{w}$ with $w$ and $\bar{w}$ extremal and $w$ having divisible value group,
does it follow that $v$ is extremal? \sn
{\bf 2)} \ We know that if $v=w\circ\bar{w}$ is extremal, then so is $\bar{w}$ (see Lemma~\ref{vexbarwex} below).
But we do not know whether it follows that also $w$ is extremal.

\parm
Tame fields of positive residue characteristic $p>0$ are
algebraically complete, and by \cite[Theorem~3.2]{[K4]}, they have
$p$-divisible value groups which consequently are not
$\mathbb{Z}$-groups. On the other hand, by the same theorem all
algebraically complete valued fields with divisible value group and
perfect residue field are tame fields. Therefore, in the case of
positive residue characteristic and value groups that are not
$\mathbb{Z}$-groups, the above Theorem~\ref{gen} is in fact talking
about tame fields:

\begin{theorem}                             \label{taex}
A tame field of positive residue characteristic is extremal if and only
if its value group is divisible and its residue field is large.
\end{theorem}

Again, we see that we know almost everything about tame fields (with the
exception of quantifier elimination in the case of equal
characteristic), but almost nothing about imperfect valued fields. As
shown in \cite{[AKP]}, there are some algebraically complete valued
fields with value group a $\mathbb{Z}$-group and a finite residue field
that are extremal, and others that are not. In particular, the Laurent
series field $\F_q((t))$ over a finite field $\F_q$ with $q$ elements is
extremal.

It is a longstanding open question whether $\F_q((t))$ has a
decidable elementary theory. However, in recent years progress has been made
on the existential theory. Denef and Schoutens showed in \cite{[DS]} that if
Resolution of Singularity holds in positive characteristic in all dimensions
(which is a longstanding open problem), then the existential theory of
$(\F_q((t)),t)$ --- i.e., the field together
with the constant $t$ --- is decidable. More recently, Anscombe and Fehm
showed in \cite{[AF]} that the existential theory of $\F_q((t))$ is decidable,
under no assumptions.

Since the question for the full elementary theory has remained open,
it is important to search for a complete
recursive axiomatization. Such an axiomatization was suggested in
\cite{[K1]}, using the elementary property that the images of additive
polynomials have the optimal approximation property (see
Section~\ref{sectoap} for the definition of this notion). For the case
of $\F_q((t))$, this was proved in \cite{[DK]}. At first sight,
extremality seems to imply the optimal approximation property for the
images of additive polynomials. But the latter uses inputs from the
whole field while the former restricts to inputs from the valuation
ring. However, we will prove in Section~\ref{sectoap}:

\begin{theorem}                             \label{OAP}
If $(K,v)$ is an extremal field of characteristic $p>0$ with
$[K:K^p]<\infty$, then the images of all additive polynomials have the
optimal approximation property.
\end{theorem}
\sn
{\bf Open problem:} \sn
Does the assertion of this theorem fail in the case of $[K:K^p]=\infty$?

\parm
Since the elementary property of extremality is more comprehensive and
easier to formulate than the optimal approximation property, it is
therefore a good idea to replace the latter by the former in the
proposed axiom system for $\F_q((t))$. We also note that every extremal
field is algebraically complete by Theorem~\ref{gen}. So we ask:

\sn
{\bf Open problem:} Is the following axiom system for the
elementary theory of $\F_q((t))$ complete?
\sn
1) $(K,v)$ is an extremal valued field of positive characteristic,
\n
2) $vK$ is a $\Z$-group,
\n
3) $Kv=\F_q\,$.

\pars
In order to obtain the assertion of Theorem~\ref{OAP} in the case of
algebraically complete perfect fields of positive characteristic (which
are exactly the tame fields of positive characteristic), one does not
need the assumption that the field be extremal. Indeed, S.~Durhan
recently proved in \cite{[D]}:

\begin{theorem}
If $(K,v)$ is a tame field of positive characteristic, then the images
of all additive polynomials have the optimal approximation property.
\end{theorem}

There are tame fields of positive characteristic that are not extremal,
e.g.\ the power series field $\F_p((\Gamma))$ with $\Gamma$ the $p$-divisible hull
of $\Z$ (see Theorem~\ref{taex}). Therefore, the previous theorem yields:

\begin{corollary}
There are perfect non-extremal fields of positive characteristic in which the images
of all additive polynomials have the optimal approximation property.
\end{corollary}
\sn
{\bf Open problem:} \sn
Is there an imperfect non-extremal field of characteristic $p>0$ in which the images
of all additive polynomials have the optimal approximation property?

\parb
Finally, let us point out that we still do not have a complete
characterization of extremal fields:
\sn
{\bf Open problem:} Take a valued field $(K,v)$ of positive residue
characteristic. Assume that $vK$ is a $\Z$-group, or that $vK$ is
divisible and $Kv$ is an imperfect large field. Under which additional
assumptions do we obtain that $(K,v)$ is extremal?

\pars
Additional assumptions are indeed needed, as we will show in
Section~\ref{sectexextr}:

\begin{proposition}                         \label{exextr}
\n
a) \ There are algebraically complete valued fields $(K,v)$ of positive
characteristic and value group a $\Z$-group that are extremal, and
others that are not.
\sn
b) \ There are algebraically complete valued fields $(K,v)$ of mixed
characteristic with value group a $\Z$-group that are extremal, and
others that are not.
\sn
c) \ There are algebraically complete nontrivially
valued fields $(K,v)$ of positive characteristic with divisible value
group and imperfect large residue field that are extremal, and
others that are not.
\sn
d) \ There are algebraically complete valued fields $(K,v)$ of mixed
characteristic with divisible value group and imperfect large residue
field that are extremal, and others that are not.
\end{proposition}

None of the non-extremal fields that we construct for the proof of parts a)--d) of this proposition is maximal.
This leads us to the following\sn
{\bf Conjecture:} \ Every maximal field with value group a $\Z$-group, or divisible value group and large residue field, is extremal.

\pars
The following theorem, also proved in
Section~\ref{sectexextr}, provides a compelling way of constructing maximal extremal fields and is used
in the proof of parts c) and d) of the previous theorem.

\begin{theorem}\label{thm:extremal.from.saturated}
Let $(K,v)$ be any $\aleph_{1}$-saturated valued field. Assume that $\Gamma$ and $\Delta$ are convex subgroups of
$vK$ such that $\Delta\subsetuneq\Gamma$ and $\Gamma/\Delta$ is archimedean. Let $u$ (respectively $w$) be the
coarsening of $v$ corresponding to $\Delta$ (resp. $\Gamma$). Denote by $\bar{u}$ the valuation induced on $Kw$ by
$u$. Then $(Kw,\bar{u})$ is maximal, extremal and large, and its value group is isomorphic either to $\Z$ or to $\R$.
In the latter case, also  $Ku=(Kw)\bar{u}$ is large.
\end{theorem}

\begin{remark}\label{rem:gamma.delta}
Pairs $(\Gamma,\Delta)$ of convex subgroups satisfying the conditions of this theorem are abundant and can easily
be constructed. Indeed, for any $\gamma\in vK$ we can take $\Gamma$ to be the smallest convex subgroup of $vK$
containing $\gamma$ (the intersection of all convex subgroups of $vK$ containing $\gamma$), and $\Delta$ to be the
largest convex subgroup of $vK$ not containing $\gamma$ (the union of all convex subgroups of $vK$ not containing
$\gamma$). Then $\Delta$ is the largest proper convex subgroup of $\Gamma$ and therefore, $\Gamma/\Delta$ is
archimedean.
\end{remark}
\pars
From Theorem~\ref{thm:extremal.from.saturated} we can derive an interesting observation about infinite
compositions of henselian valuations.
Note that every valuation can be viewed as a possibly infinite composition of rank 1 valuations, i.e., valuations
with archimedean ordered value groups. It is well known that $v=w\circ\bar{w}$ is henselian if and only if both $w$
and $\bar{w}$ are. However, in Section~\ref{sectexextr} we will derive the following result:

\begin{corollary}                                        \label{non-comp}
There exist non-large (and therefore non-henselian) valued fields $(K,v)$ with the following property: if $v=w_1
\circ w_2\circ w_3$ with $w_2$ of rank 1, then $w_2$ is henselian and both $Kw_1$ and $(Kw_1)w_2$ are large.
\end{corollary}
The part about henselianity also follows from an actually stronger result, stating the existence of a non-henselian
valued field $(K,v)$ with the following property: if $v=w_1\circ w_2$ with nontrivial $w_1$, then $w_2$ is
henselian; see \cite[Proposition 4]{[K3a]}. The latter again implies that $Kw_1$ is large, but we do not
know how to show that the field constructed in the cited paper is not large.

\bn\bn
%\newpage\noindent
{\bf Acknowlegements.} \n
Several of the ideas contained in this paper were conceived at a 2 hour seminar
talk the second author gave to the logic group at the University of Wroclaw in
Poland. The audience was arguably the most lively and inspiring the
author has ever witnessed. He would like to thank this group for the
great hospitality.
\pars
The authors would like to thank the referee for his careful reading of the manuscript and for several
very useful suggestions that inspired them to come up with
Theorem~\ref{thm:extremal.from.saturated} and with a new version of Theorem~\ref{tmcne}.
\pars
The second author would like to thank Koushik Pal for proofreading an earlier version of the paper, and Anna Blaszczok for very helpful corrections and comments on a later version.
\pars
During this research, the first author was funded by EPSRC grant EP/K020692/1, and the second author
was partially supported by a Canadian NSERC grant and a sabbatical grant from the University of Saskatchewan.

%
%???????????????????????????????????????????????????????????????????????
%
\section{Proof of Theorems~\ref{gen},~\ref{triangle} and~\ref{tmcne},
and Proposition~\ref{compimpr}}             \label{sect2}
As a preparation, we need a few basic facts about tame fields. For the
following lemma, see \cite[Lemma~3.7]{[K4]}:

\begin{lemma}                               \label{trac}
Take a tame field $(L,v)$. If $K$ is a relatively algebraically closed
subfield of $L$ such that $Lv|Kv$ is algebraic, then $(K,v)$ is a tame
field, $vL/vK$ is torsion free, and $Lv=Kv$.
\end{lemma}

We derive:

\begin{corollary}                               \label{td}
Take a tame field $(K,v)$ and an ordered abelian group $\Gamma\subset
vK$ such that $vK/\Gamma$ is torsion free. Then there exists a tame
subfield $(K',v)$ of $(K,v)$ with $vK'=\Gamma$ and $K'v=Kv$.
\end{corollary}
\begin{proof}
%We start by constructing a tame subfield $K'$ of $K$ such that $vK=\Q$
%and $K'v=Kv$.
Denote the prime field of $K$ by $K_0$ and note that $k_0:=K_0v$ is the
prime field of $Kv$. Take a maximal system $\gamma_i$, $i\in I$, of
elements in $\Gamma$ rationally independent over $vK_0\,$. Choose
elements $x_i\in K$ such that $vx_i=\gamma_i\,$, $i\in I$. Further, take
a transcendence basis $t_j\,$, $j\in J$, of $Kv$ over its prime field,
and elements $y_j\in K$ such that $y_j v=t_j$ for all $j\in J$. For
$K_1:=K_0(x_i,y_j\mid i\in I\,,\,j\in J)$ we obtain from
\cite[Lemma~2.2]{[K4]} that $vK_1=vK\oplus \bigoplus_{i\in I}^{}
\gamma_i\Z$ and $K_1v= k_0(t_j\mid j\in J)$, so that $\Gamma/vK_1$ is a
torsion group and $Kv|K_1v$ is algebraic.

Now we take $K'$ to be the relative algebraic closure of $K_1$ in $K$.
Then by Lemma~\ref{trac}, $(K',v)$ is a tame field with $vK/vK'$
torsion free and $K'v=Kv$. Since $\Gamma\subseteq vK$ and $\Gamma/vK_2$
is a torsion group, we have that $\Gamma\subseteq vK'$. Since
$vK/\Gamma$ is torsion free, we also have that $vK'\subseteq \Gamma$, so
that $vK'=\Gamma$.
\end{proof}

\begin{lemma}                               \label{tu}
Take a tame field $(K,v)$ and an ordered abelian group $\Delta$
containing $vK$ such that $\Delta$ is $p$-divisible, where $p$ is the
characteristic exponent of $Kv$. Then there exists a tame extension
field $(L,v)$ of $(K,v)$ with $vL=\Delta$ and $Lv=Kv$.
\end{lemma}
\begin{proof}
By Theorem~2.14 of \cite{[K2]} there is an extension $(K_1,v)$ of
$(K,v)$ such that $vK_1=\Delta$ and $K_1v=Kv$. We take $(L,v)$ to be a
maximal immediate algebraic extension of $(K_1,v)$; then $(L,v)$ is
algebraically maximal. Since $vL=vK_1=\Delta$ is $p$-divisible, and
$Lv=K_1v=Kv$ is perfect by \cite[Theorem~3.2]{[K4]} applied to $(K,v)$,
it follows from the same theorem that $(L,v)$ is a tame field.
\end{proof}

Now we can give the
\n
{\bf Proof of Theorem~\ref{triangle}:} \ Since $\Gamma\prec vK$ by assumption,
we have that $vK/\Gamma$ is torsion free. Hence by Corollary~\ref{td}
we find a tame subfield $(K',v)$ of $(K,v)$ with $vK'=\Gamma$ and
$K'v=Kv$. Again since $\Gamma\prec vK$, it follows from Theorem~\ref{AKE} that $(K',v) \prec (K,v)$.

Since $(K',v)$ is a tame field, we know that $\Gamma=vK'$ is $p$-divisible. As $\Gamma\prec\Delta$, the same
holds for $\Delta$. Hence by Lemma~\ref{tu} we can find a tame extension field $(L,v)$ of $(K',v)$
with $vL=\Delta$ and $Lv=K'v$. Since $vK'=\Gamma\prec\Delta=vL$, it follows again from Theorem~\ref{AKE} that
$(K',v) \prec (L,v)$.                                                   \qed

\pars
Theorem~\ref{triangle} is the key to the
\n
{\bf Proof of Theorem~\ref{gen}:} \
In view of Theorems~1.2 and~4.1 of \cite{[AKP]}, we only have to show that if
$(K,v)$ is algebraically complete with divisible value group and large perfect residue
field, then $(K,v)$ is extremal. Note that $(K,v)$ is then a tame field,
being algebraically complete with perfect residue field and $p$-divisible value group.

Every trivially valued field is
extremal, so we may assume that $(K,v)$ is nontrivially valued. We apply
Corollary~\ref{tfeqtf} with $\Delta=\R$ to obtain a tame field
$(L,v)\equiv (K,v)$ with value group $vL=\R$. By the proof of
Theorem~1.2 in \cite{[AKP]}, this field is extremal. Since extremality
is an elementary property, also $(K,v)$ is extremal.           \qed

\parm
We turn to the
\n
{\bf Proof of Theorem~\ref{tmcne}:} \ We extend the $p$-adic valuation
$v_p$ of $\Q$ to some valuation $v$ on the algebraic closure of $\Q$.
Adjoining a primitive $p$-th root of unity $\zeta_p$ to $\Q$ and passing
to the henselization $K:=\Q(\zeta_p)^h=\Q^h(\zeta_p)$, we obtain that
$vK=\frac{1}{p-1} (v_p p)\Z$ and $Kv=\Q v_p=\Fp$.

By general ramification theory, the
Galois extension $\F_{p^p}|\Fp$ can be lifted to a Galois extension of degree $p$ of
$K$. Since $K$ contains the $p$-th roots of unity, Kummer theory shows that this extension is generated
by an arbitrary $p$-th root of some element $b\in K$.
%Since $K$ lies in the completion of $\Q(\zeta_p)$, we can approximate $b'$ arbitrarily well by an element
%$b\in \Q(\zeta_p)$, and Krasner's Lemma then shows that any $p$-th root of $b$ generates the same Galois
%extension of $K$.

Now we take $L_0$ (respectively, $F_0$) to be the Galois extension of $K$ generated by a $p$-th root of $bp$ (resp., of $p$). Then $\frac{1}{p(p-1)} (v_p p)\Z\subseteq vL_0$ and $\frac{1}{p(p-1)} (v_p p)\Z\subseteq vF_0$, and since \[
[L_0:K]\>=\>[F_0:K]\>=\>p\>=\>\left(\frac{1}{p(p-1)} (v_p p)\Z\,:\,vK\right)\>,
\]
the fundamental inequality $n\geq ef$ shows that
\[
vF_0\>=\>vL_0\>=\>\frac{1}{p(p-1)} (v_p p)\Z \;\;\mbox{ and }\;\; L_0v\>=\>F_0v\>=\>\F_p\>.
\]

Since both $(L_0,v)$ and $(F_0,v)$ are henselian fields of characteristic 0 with value group isomorphic to $\Z$, they are algebraically complete. Hence by \cite[Theorem~4.1]{[AKP]}, both fields are extremal.

\parm
Next, in order to construct $(L_1,v)$ and $(F_1,v)$, we choose algebraic extensions $(L'_1,v)$ of $(L_0,v)$ and $(F'_1,v)$ of $(F_0,v)$ such that $vL'_1=vF'_1$ is the $p$-divisible hull of
$vL_0=vF_0$ and hence of $\frac{1}{p-1} (v_p p)\Z$, and
$L'_1v =L_0v=\Fp=F_0v=F'_1v$; this is possible by \cite[Theorem 2.14]{[K2]}.

Now we take $(L_1,v)$ (resp., $(F_1,v)$) to be a maximal immediate algebraic extension of
$(L'_1,v)$ (resp., $(F'_1,v)$). Then by \cite[Theorem 3.2]{[K4]}, $(L_1,v)$ and $(F_1,v)$ are
tame fields. Their value groups and residue fields are as in the assertion of
part b) of Theorem~\ref{tmcne}.

\parm
Finally, in order to construct $(L_2,v)$ and $(F_2,v)$, we choose an arbitrary non-trivially valued henselian and
perfect field $(k,w)$ of characteristic $p$ such that $kw=\F_p\,$. (For example, we could take the power series field $\F_p((t^{\Q}))$
for $k$ and the $t$-adic valuation for $w$; but also the much smaller relative algebraic closure of $\F_p(t)$ in
$\F_p((t^{\Q}))$ works.) Using \cite[Theorem 2.14]{[K2]} again, we construct extensions $(L'_2,v)$ of $(L_1,v)$
and $(F'_2,v)$ of $(F_1,v)$ such that $vL'_2=vF'_2=\Q$ and $L'_2 v=F'_2 v=k$. As before, we take $(L_2,v)$
(resp., $(F_2,v)$) to be a maximal immediate algebraic extension of $(L'_2,v)$ (resp., $(F'_2,v)$). Then again by
\cite[Theorem 3.2]{[K4]}, $(L_2,v)$ and $(F_2,v)$ are tame fields. Since their residue field $k$ admits a nontrivial
henselian valuation, it is a large field. Hence by Theorem~\ref{gen}, $(L_2,v)$ and $(F_2,v)$ are also extremal.

\parm
It remains to show that $(L_i,v)$ and $(F_i,v)$ are not elementarily
equivalent, for $i=1,2,3$. Assume the contrary. Then $L_0$ and $F_0$ or $L_1$ and $F_1$
would be isomorphic over $\Q$, as all of them are algebraic over $\Q$. Likewise, if $(L_2,v)$ and $(F_2,v)$
are elementarily equivalent then we obtain an isomorphism of the algebraic
parts of $L_2$ and $F_2$ over $\Q$. In all three cases, this yields an embedding of $F_0$ in $L_2$ and hence
the existence of all $p$-th roots of $p$ in $L_2\,$. But $L_2$ also contains a $p$-th root
of $bp$, hence a $p$-th root of $b$ as well. This however contradicts the fact that by
construction, $(L_2v)w$ does not contain $\F_{p^p}\,$.            \qed

\parm
We conclude this section with the \n
{\bf Proof of Proposition~\ref{compimpr}:} \ In both directions we assume that Kv is perfect.

First we assume that $(K,v)$
is extremal and $vK$ is divisible. By the compatibility of ``divisible value group''
with composition, both $wK$ and $\bar{w}(Kw)$ are divisible. Theorem~\ref{gen} shows that
$(K,v)$ is algebraically complete and that $Kv=(Kw)\bar{w}$ is large.
By the compatibility of ``algebraically complete''
with composition, both $(K,w)$ and $(Kw,\bar{w})$ are algebraically complete. The latter
has a large perfect residue field, hence by Theorem~\ref{gen}, it is extremal.
As in addition its value group is divisible and its residue field is perfect, it is itself perfect.
Since $Kw$ carries the nontrivial henselian valuation $\bar{w}$, it is large
(see e.g.\ \cite[Proposition 16]{[K2a]}). Therefore, also
$(K,w)$ has a large perfect residue field, and again it follows from Theorem~\ref{gen}
that it is extremal.

For the converse, we assume that both $(K,w)$ and $(Kw,\bar{w})$ are extremal with divisible
value group. By Theorem~\ref{gen}, both are algebraically complete, with large residue fields.
By compatibility it follows that $(K,v)$ is algebraically complete with divisible value group.
We know that $Kv=(Kw)\bar{w}$ is large, and it is also perfect by assumption. Now
Theorem~\ref{gen} shows that $(K,v)$ is extremal.

%
%???????????????????????????????????????????????????????????????????????
%
\section{Additive polynomials over extremal fields}\label{sectoap}
We start by introducing a more precise notion of extremality. Take a
valued field $(K,v)$, a subset $S$ of $K$, and a polynomial $f$ in $n$
variables over $K$. Then we say that \bfind{$(K,v)$ is $S$-extremal
with respect to $f$} if the set $vf(S^n)\subseteq vK\cup\{\infty\}$ has
a maximum. We say that \bfind{$(K,v)$ is $S$-extremal} if it is
$S$-extremal with respect to every polynomial in any finite number of
variables. With this notation, $(K,v)$ being extremal means that it is
${\cal O}$-extremal, where ${\cal O}$ denotes the valuation ring of
$(K,v)$.

A subset $A$ of a valued field $(K,v)$ has the {\bf optimal
approximation property} if for every $z\in K$ there is some $y\in A$
such that $v(z-y)= \max\{v(z-x)\mid x\in A\}$. A polynomial $h\in
K[X_1,\ldots,X_n]$ is called a \bfind{$p$-polynomial} if it is of the
form $f+c$, where $f\in K[X_1,\ldots,X_n]$ is an additive polynomial and
$c\in K$. The proof of the following observation is straightforward:

\begin{lemma}                               \label{oapextr}
The images of all additive polynomials over $(K,v)$ have the optimal
approximation property if and only if $K$ is $K$-extremal with respect
to all $p$-polynomials over $K$.
\end{lemma}

We will work with ultrametric balls
\[
B_\alpha(a)\>:=\>\{b\in K\mid v(a-b)\geq\alpha\}\>,
\]
where $\alpha\in vK$ and $a\in K$. Observe that ${\cal O}=B_0(0)$. We
note:

\begin{proposition}                         \label{BaBb}
Take $\alpha,\beta\in vK$ and $a,b\in K$. Then $(K,v)$ is
$B_\alpha(a)$-extremal if and only if it is $B_\beta(b)$-extremal. In
particular, $(K,v)$ is $B_\alpha(a)$-extremal if and only if it is
extremal.
\end{proposition}
\begin{proof}
It suffices to prove that ``$B_\alpha(a)$-extremal'' implies
``$B_\beta(b)$-extremal''. Take a polynomial $f$ in $n$ variables. If
$c\in K$ is such that $vc=\beta-\alpha$, then the function $y\mapsto
c(y-a)+b$ establishes a bijection from $B_\alpha(a)$ onto $B_\beta(b)$.
We set $g(y_1,\ldots,y_n):=f(c(y_1-a)+b,\ldots,c(y_n-a)+b)$. It follows
that $f(B_\beta(b)^n)=g(B_\alpha(a)^n)$, whence $vf(B_\beta(b)^n)=
vg(B_\alpha(a))^n$. Hence if $(K,v)$ is $B_\alpha(a)$-extremal with
respect to $g$, then it is $B_\beta(b)$-extremal with respect to $f$.
This yields the assertions of the proposition.
\end{proof}

A valued field $(K,v)$ of characteristic $p>0$ is called
\bfind{inseparably defectless} if every finite purely inseparable
extension $(L|K,v)$ satisfies equation~(\ref{dle}) (note that the
extension of $v$ from $K$ to $L$ is unique). This holds if and only if
every finite subextension of $(K|K^p,v)$ satisfies equation~(\ref{dle}).

If $(K,v)$ is inseparably defectless with $[K:K^p]<\infty$, then for
every $\nu\geq 1$, the extension $(K|K^{p^{\nu}},v)$ has a
\bfind{valuation basis}, that is, a basis of elements $b_1,\ldots,b_\ell$
that are \bfind{valuation independent} over $K^{p^{\nu}}$, i.e.,
\[
v(c_1b_1+\ldots+c_\ell b_\ell)\>=\>\min_{1\leq i\leq\ell} vc_ib_i
\]
for all $c_1,\ldots,c_\ell\in K^{p^{\nu}}$.

Note that every algebraically complete valued field is in particular
inseparably defectless. By Theorem~\ref{gen}, every extremal field is
algebraically complete and hence inseparably defectless.

\begin{proposition}                               \label{gi}
Take an inseparably defectless valued field $(K,v)$ with $[K:K^p]<
\infty$ and an additive polynomial $f$ in $n$ variables over $K$. Then
for some integer $\nu\geq 0$ there are additive polynomials
$g_1,\ldots,g_m\in K[X]$ in one variable such that
\sn
a) \ $f(K^n)=g_1(K) +\ldots+g_m(K)$,\n
b) \ all polynomials $g_i$ have the same degree $p^{\nu}$,\n
c) \ the leading coefficients $b_1,\ldots,b_m$ of $g_1,\ldots,g_m$ are
valuation independent over $K^{p^{\nu}}$.
\end{proposition}
\begin{proof}
The proof can be taken over almost literally from Lemma~4 of
\cite{[DK]}.
One only has to replace the elements $1,t,\ldots,t^{\delta_i-1}$ from
that proof by an arbitrary basis of $K|K^{\delta_i}$.
\end{proof}

The following theorem is a reformulation of Theorem~\ref{OAP} of the
Introduction.
\begin{theorem}                            % \label{OAPr}
Assume that $(K,v)$ is an extremal field of characteristic $p>0$ with
$[K:K^p]<\infty$. Then it is $K$-extremal w.r.t.\ all $p$-polynomials
and therefore, the images of all additive polynomials have the optimal
approximation property.
\end{theorem}
\begin{proof}
Take a $p$-polynomial $h$ in $n$ variables over $K$, and write it as
$h=f+c$ with $f$ an additive polynomial in $n$ variables over $K$ and
$c\in K$. We choose additive polynomials $g_1,\ldots,g_m\in K[X]$ in one
variable satisfying assertions a), b), c) of Proposition~\ref{gi}. Then
$h(K^n)=g_1(K)+\ldots+g_m(K)+c$.

We write $g_i=b_iX^{p^{\nu}}+c_{i,\nu-1}X^{p^{\nu-1}}+\ldots+c_{i,0}X$
for $1\leq i\leq m$. Then we choose $\alpha\in vK$ such that
%$\alpha<0$, $\alpha<vc$ and
\[
\alpha\><\>\min\{0,\,vc-vb_i\, ,\,vc_{i,k}-vb_i\mid 1\leq i\leq m\,,\,
0\leq k<\nu\}\>.
\]
Because $\alpha<0$, it then follows that for each $a$
with $va\leq\alpha$,
\[
vb_i+p^{\nu}va \>\leq\> vb_i+p^{\nu}\alpha \>\leq\> vb_i+\alpha \><\>vc
\]
and for $0\leq k<\nu$,
\[
vb_i+p^{\nu}va \>\leq\> vb_i+va+p^k va \>\leq\> vb_i+\alpha
+p^k va\><\> vc_{i,k}+p^k va\>.
\]
%\begin{eqnarray*}
%vb_i+p^{\nu}\alpha & < & vb_i+\alpha \><\>vc\\
%\min\{0,\,vc+(p^{\nu}-1)\alpha\, ,\,vc_{i,k}+(p^{\nu}-1)\alpha
%\mid 1\leq i\leq m\,,\, 0\leq k<\nu\} \\
% & \leq & \min\{0,\,vc\, ,\,vc_{i,k}+p^{k}\alpha
%\mid 1\leq i\leq m\,,\, 0\leq k<\nu\}
%\end{eqnarray*}
%since $\alpha<0$. If $va\leq\alpha$, this yields that
%
It then follows that
\begin{equation}                            \label{a}
vg_i(a)\>=\>vb_i+p^{\nu}va\>\leq\>vb_i+p^{\nu}\alpha \><\>vc\>.
\end{equation}
On the other hand, if $va'\geq\alpha$, then $vb_i+p^{\nu}va'\>\geq\>
vb_i+p^{\nu}\alpha$ and $vc_{i,k}+p^kva' \>\geq\> vc_{i,k}+p^k\alpha \>
>\>vb_i+p^{\nu}\alpha$ for $0\leq k<\nu$. This yields that
\begin{equation}                            \label{a'}
vg_i(a')\>\geq\>vb_i+p^{\nu}\alpha\>.
\end{equation}

Now take any $(a'_1,\ldots,a'_m)\in B_\alpha(0)^n$ and
$(a_1,\ldots,a_m)\in K^n\setminus B_\alpha(0)^n$. So we have:
\[
\min\{va_1,\ldots,va_m\}\><\>\alpha\>\leq\>\min\{va'_1,\ldots,va'_m\}\>.
\]
Since $b_1,\ldots,b_m$ are valuation independent over $K^{p^{\nu}}$, we
then obtain from (\ref{a}) and (\ref{a'}) that
\begin{eqnarray*}
vh(a_1,\ldots,a_m) & = & \min_{1\leq i\leq m} vb_i+p^{\nu}va_i\\
& < & \min_{1\leq i\leq m}vb_i+p^{\nu}\alpha\>\leq\>
vh(a'_1,\ldots,a'_m)\>.
\end{eqnarray*}
%
%On the other hand, if $a'_1,\ldots,a'_m\in B_\alpha(0)$, then
%
%$\min\{va_1,\ldots,va_m\}<\min\{va'_1,\ldots,va'_m\}$ implies
%\begin{eqnarray*}
%vh(a'_1,\ldots,a'_m) & \geq & \min\{vc,vb_i+p^{\nu}va_i,
%vc_{i,k}+p^{k} va_i\mid 1\leq i\leq m\,,\, 0\leq k<\nu\}\\
%& \geq & \min\{vc,vb_i+p^{\nu}\alpha,
%vc_{i,k}+p^{k} \alpha\mid 1\leq i\leq m\,,\, 0\leq k<\nu\}\\
%& = & \min_{1\leq i\leq m}vb_i+p^{\nu}\alpha\\
%
%\min_{1\leq i\leq m}va'_i+p^{\nu}va'_i\\
% & \geq & \min_{1\leq i\leq m}vb_i+p^{\nu}va_i\>=\>vh(a_1,\ldots,a_m)\>.
%\end{eqnarray*}
%
This proves that
\[
vh(B_\alpha(0)^n)\> >\>vh(K^n\setminus B_\alpha(0)^n)\>.
\]
Since $(K,v)$ is extremal by assumption, Proposition~\ref{BaBb} shows
that $vh(B_\alpha(0)^n)$ has a maximal element, and the same is
consequently true for $vh(K^n)$. This shows that $(K,v)$ is
$K$-extremal w.r.t.\ $h$, from which the first assertion follows. The
second assertion follows by Lemma~\ref{oapextr}.
\end{proof}

%
%???????????????????????????????????????????????????????????????????????
%
\section{More constructions of extremal fields, and proof of
Theorem~\ref{thm:extremal.from.saturated}}        \label{sectexextr}
It follows from \cite[Theorem~4.1]{[AKP]} that the Laurent series fields
$(\F_p((t)),v_t)$ and the $p$-adic fields $(\Q_p,v_p)$ are extremal. The
former have equal characteristic, the latter mixed characteristic. All
of them have $\Z$ as their value group, which is a $\Z$-group.

In \cite{[K1]} a valued field extension $(L,v)$ of $(\F_p((t)),v_t)$ is
presented in which not all images of additive polynomials have the
optimal approximation property. In \cite{[AKP]} it is shown that $(L,v)$
is not extremal, although it is algebraically complete and its value
group $vL$ is a $\Z$-group (of rank 2). It is also shown that for
the nontrivial coarsening $w$ of $v$ corresponding to the convex
subgroup $(v_t t)\Z$ of $vL$, also $(L,w)$ is not extremal. As a coarsening of an
algebraically complete valuation, it is also algebraically complete. Its value
group $wL=vL/(v_t t)\Z$ is divisible and its residue field
$Lw=\F_p((t))$ is large, but not perfect. Note that $(L,v)$ and $(L,w)$
are of equal characteristic.

\pars
In order to prove the remaining existence statements of
Proposition~\ref{exextr} concerning non-extremal fields in mixed
characteristic, we consider compositions of valuations.
Unfortunately, contrary to the assertion that the proof of Lemma~5.2 of
\cite{[AKP]} is easy (and thus left to the reader), we are unable to
prove it in the cases that are not covered by
Proposition~\ref{compimpr}. (However, we also do not know of any
counterexample.) In fact, a slightly different version can easily be
proved: {\it If $(K,v)$ is ${\cal O}_v$-extremal, then
also $(K,w)$ is ${\cal O}_v$-extremal.} We do not know whether the
latter impies that $(K,w)$ is ${\cal O}_w$-extremal. Proposition~\ref{BaBb}
is of no help here because ${\cal O}_v$ is in general not a ball of the
form $B_\alpha(a)$ in $(K,w)$.

It appears, though, that we actually had in mind
the following result, which is indeed easy to prove:

\begin{lemma}                               \label{vexbarwex}
If $(K,v)$ is extremal and $v=w\circ\ovl{w}$, then
$(Kw,\ovl{w})$ is extremal.
\end{lemma}
\begin{proof}
Assume that $(K,v)$ is extremal with $v=w\circ\ovl{w}$; note that for
any $a,b\in {\cal O}_w\,$, $\ovl{w}(aw)>\ovl{w}(bw)$ implies $va>vb$.

Assume further that $g\in Kw[X_1,\ldots,X_n]$. Then choose $f\in {\cal O}_w[X_1,\ldots,X_n]$
such that $fw=g$. By assumption, there are $b_1,\ldots,b_n\in {\cal O}_v$ such that
\[
vf(b_1,\ldots,b_n)\>=\>\max\{vf(a_1,\ldots,a_n) \mid a_1,\ldots,a_n\in
{\cal O}_v\}\>.
\]
Since $b_1,\ldots,b_n\in {\cal O}_v\subseteq {\cal O}_w$ we have that
\[
f(b_1,\ldots,b_n)w\>=\>fw(b_1w,\ldots,b_nw)\>=\>g(b_1w,\ldots,b_nw)\>.
\]
We claim that
\[
\ovl{w}g(b_1w,\ldots,b_nw)\>=\>\max\{\ovl{w}g(\ovl{a}_1,\ldots,\ovl{a}_n)
\mid \ovl{a}_1,\ldots,\ovl{a}_n\in {\cal O}_{\ovl{w}}\}\>.
\]
Indeed, if there were $\ovl{a}_1,\ldots,\ovl{a}_n\in {\cal O}_{\ovl{w}}$
with $\ovl{w}g(\ovl{a}_1,\ldots,\ovl{a}_n)>\ovl{w}g(b_1w,\ldots,b_nw)$,
then for any choice of $a_1,\ldots,a_n\in {\cal O}_w$ with
$a_iw=\ovl{a}_i$ for $1\leq i\leq n$ we would obtain that $a_1,\ldots,a_n\in {\cal O}_v$
and $vf(a_1,\ldots,a_n)>vf(b_1,\ldots,b_n)$, a contradiction.
%
%${\cal O}_v \subseteq {\cal O}_w$ and $wa>wb>0\Rightarrow va>vb>0$.
%Suppose that $\{w(f(a_1,\dots,a_n))\mid a_1,\dots a_n \in {\cal O}_w\}$
%has no maximum. This does not change if we multiply $f$ by a nonzero
%constant. So we can assume that all coefficients lie in ${\cal O}_v\,$.
%and at least one of them is
\end{proof}

\pars
We use this lemma to prove the existence of the non-extremal fields in mixed
characteristic as claimed in Proposition~\ref{exextr}. We consider again
the two non-extremal fields $(L,v)$ and $(L,w)$ mentioned above. By
Theorem~2.14 of \cite{[K2]} there is an extension $(K_0,v_0)$ of
$(\Q,v_p)$ with divisible value group and $L$ as its residue field. We
replace $(K_0,v_0)$ by a maximal immediate extension $(M,v_0)$. Then
$(M,v_0)$ is algebraically complete, and so are $(M,v_0\circ v)$ and
$(M,v_0\circ w)$. The value group of $(M,v_0\circ v)$ is a $\Z$-group,
and $(M,v_0\circ w)$ has divisible value group and nonperfect large
residue field. But by Lemma~\ref{vexbarwex}, both fields are non-extremal.

\parm
Finally, we have to prove the existence of extremal fields as stated in parts c) and d) of
Proposition~\ref{exextr}. We will employ Theorem~\ref{thm:extremal.from.saturated} which we will prove now. We note
that by Theorem~\ref{gen}, the residue field of an extremal field with divisible value group must be large. Also,
every non-trivially valued extremal field is henselian, which implies that it is itself a large field. Therefore,
it remains to prove the following assertion:
\sn
{\it Let $(K,v)$ be any $\aleph_{1}$-saturated valued field. Assume that $\Gamma$ and $\Delta$ are convex subgroups
of $vK$ such that $\Delta\subsetuneq\Gamma$ and $\Gamma/\Delta$ is archimedean. Let $u$ (respectively $w$) be the
coarsening of $v$ corresponding
to $\Delta$ (resp. $\Gamma$). Denote by $\bar{u}$ the valuation induced on $Kw$ by $u$. Then $(Kw,\bar{u})$ is
maximal and extremal, and its value group is isomorphic either to $\Z$ or to $\R$.}

\begin{proof}
Denote by $\mathcal{O}_{u}$ (resp.\ $\mathcal{O}_{w}$) the valuation ring corresponding to $u$ (resp.\ $w$). We note that the value group of $u$ is $vK/\Delta$, the value group of $w$ is $vK/\Gamma$, and we have that
\[
\mathcal{O}_v\>\subset\> \mathcal{O}_u\>\subset\> \mathcal{O}_w\>.
\]
We show first that $(Kw,\bar{u})$ is maximal. From \cite[Theorem 4]{[Ka]} we know that a valued field is maximal
if and only if every pseudo Cauchy sequence has a limit in the field. We refer the reader to \cite{[Ka]}
for an excellent introduction to the theory of pseudo Cauchy sequences (which Kaplansky calls ``pseudo-convergent
sets''). If $(a_i)_{i<\lambda}$ is any pseudo Cauchy sequence in $(Kw,\bar{u})$, then the sequence
$\bar{u}(a_{i+1}-a_i)$  in $\bar{u}(Kw)=\Gamma/\Delta$ is strictly increasing. But the cofinality of any
strictly increasing sequence in $\R$ (and hence also in any archimedean ordered abelian group) is at most $\omega$.
Therefore, it suffices to show that every pseudo Cauchy sequence $(a_i)_{i<\omega}$ in $(Kw,\bar{u})$ has a
limit. By definition, $a\in Kw$ is a limit of this sequence if and only if $\bar{u}(a-a_i)=\bar{u}(a_{i+1}-a_i)$
for all $i<\omega$.

Write $a_i=b_iw$ with $b_i\in \mathcal{O}_{w}\,$, $i<\omega$. Then the sequence $(u(b_{i+1}-b_i))_{i<\omega}$ is
strictly increasing in $vK/\Delta$. This implies that the sequence $(v(b_{i+1}-b_i))_{i<\omega}$ is
strictly increasing in $vK$. We consider the following (partial) type in countably many parameters:
\[
\{v(x-b_i)=v(b_{i+1}-b_i)\mid i<\omega\}\>.
\]
It is finitely realizable in $(K,v)$ since for $x=b_{i+1}$ we obtain that $v(x-b_j)=v(b_{j+1}-b_j)$ holds for $0\leq j\leq i$. By saturation, there is some $b\in K$ which realizes this type. Now $v(b-b_i)=v(b_{i+1}-b_i)$ implies that
\begin{eqnarray*}
w(b-b_i)&=&v(b-b_i)+\Gamma=v(b_{i+1}-b_i)+\Gamma=w(b_{i+1}-b_i)\>,\\
u(b-b_i)&=&v(b-b_i)+\Delta=v(b_{i+1}-b_i)+\Delta=u(b_{i+1}-b_i)\>.
\end{eqnarray*}
The former implies that $b\in \mathcal{O}_{w}$ as also all $b_i$ are in $\mathcal{O}_{w}$; so we can set $a:=bw$. The latter then implies that \[
\bar{u}(a-a_i)=\bar{u}(bw-b_iw)=\bar{u}((b-b_i)w)=u(b-b_i)=u(b_{i+1}-b_i)=\bar{u}(a_{i+1}-a_i)\>.
\]
This proves that $a\in Kw$ is a limit of the pseudo Cauchy sequence $(a_i)_{i<\omega}$ and shows that
$(Kw,\bar{u})$ is maximal.

\parm
Now we distinguish two cases.\sn
\underline{Case 1}: $\bar{u}(Kw)$ is isomorphic to $\Z$. In this case, it follows from the maximality that $(Kw,\bar{u})$ is algebraically complete and hence extremal \cite[by Theorem~4.1]{[AKP]}.

\mn
\underline{Case 2}: $\bar{u}(Kw)=\Gamma/\Delta$ is densely ordered. Note that since the archimedean ordered group $\Gamma/\Delta$ is embeddable in $\R$, any subset of it has coinitiality and cofinality no greater than $\aleph_{0}$.

We show that $(Kw,\bar{u})$ is extremal. The value group $\bar{u}(Kw)$ is $\Gamma/\Delta$ and $\mathcal{O}_{\bar{u}}$ is the image of $\mathcal{O}_{u}$ under the residue map $x\mapsto xw$ of $w$. For a tuple $\mathbf{a}=(a_1,...,a_m)$ from $\mathcal{O}_{w}$, we denote by $\mathbf{a}w:=(a_1w,...,a_m w)$ the corresponding tuple of residues.

Let $\bar{f}\in Kw[\mathbf{x}]$ be a polynomial in the variables $\mathbf{x}=(x_1,...,x_m)$ and let $f\in\mathcal{O}_{w}[\mathbf{x}]$ denote any lift of $\bar{f}$ so that $fw=\bar{f}$. We must show that the set of $\bar{u}$-values of the image of $\bar{f}$, i.e.,
\begin{align*}
X:=&\left\{\bar{u}(\bar{f}(\mathbf{b}))\in\Gamma/\Delta\cup\{\infty\}\;\left|\;\mathbf{b}\in\mathcal{O}_{\bar{u}}\right.\right\}\\
=&\left\{\bar{u}(\bar{f}(\mathbf{a}w))\in\Gamma/\Delta\cup\{\infty\}\;\left|\;\mathbf{a}\in\mathcal{O}_{u}\right.\right\},
\end{align*}
has a maximum. As noted above, the cofinality of $X$ is no greater than $\aleph_{0}$. Thus there is a sequence $(\mathbf{a}_{n})_{n<\omega}$ of $m$-tuples from $\mathcal{O}_{u}$ such that the sequence
$$(\bar{u}(\bar{f}(\mathbf{a}_{n}w))_{n<\omega}$$
is increasing and cofinal in $X$. For each $n<\omega$ we set $\alpha_{n}:=v(f(\mathbf{a}_{n}))$, and note that either $\bar{f}(\mathbf{a}_{n}w)=0$ (in which case $\bar{u}(\bar{f}(\mathbf{a}_{n}w))=\infty$ must be the maximum of $X$) or
$$\alpha_{n}+\Delta=u(f(\mathbf{a}_{n}))=\bar{u}(\bar{f}(\mathbf{a}_{n}w)).$$

Next we set $Y:=\{\gamma+\Delta\in\Gamma/\Delta\;|\;\gamma+\Delta<\Delta\}$. Then $Y$ is equal to the image under $u$ of the elements of $\mathcal{O}_{w}\setminus\mathcal{O}_{u}$ (and also equal to the image under $\bar{u}$ of $Kw\setminus\mathcal{O}_{\bar{u}}$). By assumption, $\Gamma/\Delta$ is densely ordered; thus $Y$ has no maximum. Also the cofinality of $Y$ can be no greater than $\aleph_{0}$. Thus there is a sequence $(\beta_{n})_{n<\omega}$ in $\Gamma$ such that $(\beta_{n}+\Delta)_{n<\omega}$ is a strictly increasing and cofinal sequence in $Y$.

Finally we consider the following (partial) $\mathbf{x}$-type in countably many parameters:
\[
p(\mathbf{x}):=\left\{\alpha_{n}\leq v(f(\mathbf{x}))\;\left|\;n<\omega\right.\right\}\cup\left\{\beta_{n}\leq v(x_{i})\;\left|\;n<\omega, 1\leq i\leq m\right.\right\}\>.
\]
This is finitely realised in $(K,v)$. By saturation, it is realized by some $m$-tuple $\mathbf{c}=(c_1,...,c_m) \in K^m$.

For $1\leq i\leq m$ we examine the second set of formulas in $p(\mathbf{x})$ to find that $\beta_{n}\leq v(c_{i})$, for each $n<\omega$. Thus $\beta_{n}+\Delta\leq v(c_{i})+\Delta=u(c_{i})$, again for each $n<\omega$. By the cofinality of the sequence $(\beta_{n}+\Delta)_{n<\omega}$ in $Y$ we have that $c_{i}\in\mathcal{O}_{u}$.

Finally, by examining the first set of formulas in $p(\mathbf{x})$, we see that $\alpha_{n}\leq v(f(\mathbf{c}))$, for all $n<\omega$. Then either $\bar{u}(\bar{f}(\mathbf{c}w))=\infty$ (in which case $\infty$ is the maximum of $X$) or we have that
\[
\alpha_{n}+\Delta\>\leq\> v(f(\mathbf{c}))+\Delta\>=\>u(f(\mathbf{c}))\>=\>\bar{u}(\bar{f}(\mathbf{c}w)),
\]
for all $n<\omega$. Since $(\alpha_{n}+\Delta)$ is cofinal in $X$, $\bar{u}(\bar{f}(\mathbf{c}w))$ is the maximum of $X$. This shows that $(Kw,\bar{u})$ is extremal, as required.

\parm
For the conclusion of the proof, we show that the value group of $(Kw,\bar{u})$ is cut complete, which shows that it is isomorphic to $\R$. Take a Dedekind cut $(D,E)$ in $\bar{u}(Kw)$, that is, $D$ is a nonempty initial segment of $\bar{u}(Kw)$ and $E$ is a nonempty final segment of $\bar{u}(Kw)$ such that $D\cup E=\bar{u}(Kw)$. As noted before,
the cofinality of $D$ and the coinitiality of $E$ are no greater than $\aleph_{0}$. Thus there are sequences $(\beta_{n})_{n<\omega}$ and $(\gamma_{n})_{n<\omega}$ in $\Gamma$ such that $(\beta_{n}+\Delta)_{n<\omega}$ is an increasing and cofinal sequence in $D$ and $(\gamma_{n}+\Delta)_{n<\omega}$ is a decreasing and coinitial sequence in $E$.
We consider the following (partial) type in countably many parameters:
\[
\left\{\beta_{n}\leq vx\mid n<\omega\right\}\cup\left\{\gamma_{n}\geq vx\mid n<\omega\right\}\>.
\]
This is finitely realized in $(K,v)$. Hence by saturation, it is realized by some $d\in K$. Then $\beta_{n}\leq vd\leq\gamma_n$ and therefore $\beta_{n}+\Delta\leq ud\leq\gamma_n+\Delta\,$, for each $n<\omega$. It follows that $ud$ lies in the convex hull of $\Gamma/\Delta$ in $vK/\Delta$, which shows that $wd=0$. So $dw\in Kw$, and we obtain that
\[
D\leq \bar{u}(dw)=ud \leq E\>,
\]
which proves that the cut $(D,E)$ is realized in $\bar{u}(Kw)$, showing that this group is cut complete.
\end{proof}

\pars
We may choose $(K, v)$ so that $\Gamma/\Delta$ is densely ordered, for any $\Delta\subset\Gamma\subset vK$. Indeed, if we take any integer $n\geq 2$ and $(K,v)$ such that $vK$ is $n$-divisible, then also $\Gamma/\Delta$ will be $n$-divisible and hence densely ordered. If on the other hand, the residue field $Kv$ is imperfect and $w\subset u\subset v$
are as in the theorem, then also the residue field of $(Kw,\bar{u})$, which is equal to $Ku$, is imperfect. Taking $(K,v)$ to be an $\aleph_1$-saturated valued field of equal characteristic $p$ with imperfect large residue field and $n$-divisible value group, and choosing $\Gamma$ and $\Delta$ according to Remark~\ref{rem:gamma.delta}, we obtain from Theorem~\ref{thm:extremal.from.saturated}:

\begin{corollary}\label{cor:positive.char}
Let $p$ be a prime. There exist extremal fields of equal characteristic $p$ with value group isomorphic to $\R$ and imperfect residue field.
\end{corollary}

To give an example of an extremal field obtained by this corollary, we begin by taking any $\aleph_1$-saturated elementary extension $(K,v)$ of the Puiseux series field $\bigcup_{n\in\N} \Fp(x)((t^{1/n}))$ over $\Fp(x)$, where $x$ is transcendental over $\Fp\,$. As the residue field $Kv$ is an elementary extension of the lower residue field, it is also imperfect. As the value group $vK$ is an elementary extension of the lower value group, it is also divisible.

\pars
On the other hand, we can extend the $p$-adic valuation from $\Q$ to a valuation $v$ on $\Q(x)$ such that $xv$ is transcendental over $\Fp$; then the residue field of $(\Q(x),v)$ will be the imperfect field $\Fp(xv)$. By adjoining $n$-th roots repeatedly, we can pass, without changing the residue field, to an algebraic extension $(k,v)$ of $(\Q(x),v)$ with $n$-divisible value group. Now we can take any $\aleph_1$-saturated elementary extension $(K,v)$ of $(k,v)$. Then $Kv$ will again be imperfect, $vK$ will be $n$-divisible, and $(K,v)$ will have mixed characteristic $(0,p)$.

In order to achieve that the valued field $(Kw,\bar{u})$ in Theorem~\ref{thm:extremal.from.saturated} also has mixed characteristic, we choose $\Gamma$ and $\Delta$ as follows. We take $\Delta$ to be the largest convex subgroup of $vK$ not containing $vp$ and let $\Gamma$ be the smallest convex subgroup of $vK$ containing $vp$. Then $\Delta$ is the largest proper convex subgroup of $\Gamma$, and therefore $\Gamma/\Delta$ is archimedean. It follows that $pw\ne 0$ since $vp\notin\Delta$, but $(pw)\bar{u}=pu=0$ since $vp\in\Gamma$. This shows that $\chara Kw=0$ and $\chara (Kw)\bar{u}=p$. We thus obtain:

\begin{corollary}\label{cor:mixed.char}
Let $p$ be a prime. There exist extremal fields of mixed characteristic $(0,p)$ with value group isomorphic to $\R$ and imperfect residue field.
\end{corollary}

Corollaries~\ref{cor:positive.char} and~\ref{cor:mixed.char} together complete the proof of Proposition~\ref{exextr}.

\pars
By taking $(K,v)$ as in one of these corollaries and $(L,v)$ to be a countable model of $\Th(K,v)$, we obtain:

\begin{corollary}
Let $p$ be a prime. There exist countable extremal fields of equal characteristic $p$ with divisible value group not isomorphic to $\R$ and imperfect residue field. Likewise, there exist countable extremal fields of mixed characteristic $(0,p)$ with divisible value group not isomorphic to $\R$ and imperfect residue field.
\end{corollary}

By choosing models of arbitrary cardinality, one can obtain divisible value groups of arbitrarily large cardinality. But we do not know which divisible ordered abelian groups (and not even which cardinalities) can be thus obtained, as we are lacking an AKE-principle.

\pars
We will now give the
\n
{\bf Proof of Corollary~\ref{non-comp}:} \
We take $(K,v)$ to be an $\aleph_{1}$-saturated elementary extension of an arbitrary non-large valued field whose
value group is divisible by some $n\geq 2$, and apply Theorem~\ref{thm:extremal.from.saturated}. Since also $vK$ is
divisible by $n$, for all $u$ and $w$ as in the theorem the value group of $(Kw,\bar{u})$ is divisible. Hence if $v=w_1\circ w_2\circ w_3$ with $w_2$ of rank 1, then by setting $w=w_1$ and $u=w_1\circ w_2$ it follows from the theorem that $(Kw_1,w_2)$ is extremal with nontrivial divisible value group, hence\n
a) \ $w_2$ is henselian,\n
b) \ $Kw_1$ is large,\n
c) \ $(Kw_1)w_2$ is large.
\qed

\parb
For the conclusion of this paper, let us discuss how the property of
extremality behaves in a valued field extension $(L|K,v)$ where $(K,v)$
is existentially closed in $(L,v)$. In this case, it is
known that $L|K$ and $Lv|Kv$ are regular extensions and that $vL/vK$ is
torsion free. (An extension $L|K$ of fields is called \bfind{regular} if
it is separable and $K$ is relatively algebraically closed in $L$.)

\begin{proposition}
Take a valued field extension $(L|K,v)$ such that $(K,v)$ is
existentially closed in $(L,v)$, a subset $S_K$ of $K$ that is existentially definable
with parameters in $K$, and a polynomial $f$ in $n$ variables over $K$.
Denote by $S_L$ the subset of $L$ defined by the existential formula that
defines $S_K$ in $K$. Then the following assertions hold.
\n
a) \ If $(K,v)$ is $S_K$-extremal w.r.t.\ $f$, then $(L,v)$ is
$S_L$-extremal w.r.t.\ $f$ and $\max vf(S_L^n)=\max vf(S_K^n)$. In
particular, if $(K,v)$ is extremal, then $(L,v)$ is extremal w.r.t.\ all
polynomials with coefficients in $K$.
\n
b) \ Assume in addition that $vL=vK$. If $(L,v)$ is $S_L$-extremal
w.r.t.\ $f$, then $(K,v)$ is $S_K$-extremal w.r.t.\ $f$ and $\max
vf(S_L^n)=\max vf(S_K^n)$. In particular, if $(L,v)$ is extremal, then
so is
$(K,v)$.
\end{proposition}
\begin{proof}
a): \ Assume that $a\in S_K^n$ such that $vf(a)=\max vf(S_K^n)$. Then
the assertion that there exists an element $b$ in $S_L^n$ such that
$vf(b)> vf(a)$ is an elementary existential sentence with parameters in
$K$. Hence if it held in $L$, then there would be an element $b'$ in
$S_K^n$ such that $vf(b')>vf(a)$, which is a contradiction to the choice
of $a$. It follows that $\max vf(S_L^n)\leq\max vf(S_K^n)$. Since
$S_K\subseteq S_L\,$, we obtain that $\max vf(S_L^n)=\max vf(S_K^n)$.
\sn
b): \ Take $b\in S_L^n$ such that $vf(b)=\max vf(S_L^n)$. Since
$vL=vK$ by assumption, there is $c\in K$ such that $vc=vf(b)$. Now
the assertion that there exists an element $b$ in $S_L^n$ such that
$vf(b)=vc$ is an elementary existential sentence with parameters in $K$.
Hence there is $a\in S_K^n$ such that $vf(a)=vc=\max vf(S_L^n)$. Since
$vf(a)\in vf(S_K^n)\subseteq vf(S_L^n)$, we obtain that $vf(a)=
\max vf(S_K^n)$.
\end{proof}

%\adresse

\begin{thebibliography}{99}
\bibitem{[AF]} Anscombe, S.\ -- Fehm, A.: {\it The existential theory of
equicharacteristic henselian valued fields}, http://arxiv.org/abs/1501.04522 (2015)

\bibitem{[AKP]} Azgin, S.\ -- Kuhlmann, F.-V.\ -- Pop, F.:
{\it Characterization of Extremal Valued Fields}, Proc.\ Amer.\ Math.\
Soc.\ {\bf 140} (2012), 1535--1547

\bibitem{[DK]} van den Dries, L.\ -- Kuhlmann, F.-V.: {\it Images of
additive polynomials in $\F_q((t))$ have the optimal approximation
property}, Can.\ Math.\ Bulletin {\bf 45} (2002), 71--79


\bibitem{[DS]} Denef, J.\ -- Schoutens, H.: {\it On the decidability of the existential theory of $F_p[[t]]$},
Valuation theory and its applications, Vol.\ II (Saskatoon, SK, 1999), 43?60,
Fields Inst.\ Commun., 33, Amer.\ Math.\ Soc., Providence, RI, 2003

\bibitem{[D]} Durhan, S.: {\it Additive Polynomials over Perfect
Fields}, in:  Valuation Theory in Interaction,
Proceedings of the Second International Valuation Theory Conference, Segovia / ElEscorial 2011,
EMS Series of Congress Reports 2014

\bibitem{[Ka]} Kaplansky, I.: {\it Maximal fields with valuations I},
Duke Math.\ Journ.\ {\bf 9} (1942), 303--321

%\bibitem{[Ku1]} Kuhlmann, F.-V.: {\it Henselian function fields and
%tame fields}, extended version of Ph.D. thesis, Heidelberg (1990)
%\bibitem{[Ku2]} Kuhlmann, F.-V.: Valuation theory. Book in
%preparation. Preliminary versions of several chapters are available on
%the web site http://math.usask.ca/$\,\tilde{ }\,$fvk/Fvkbook.htm

\bibitem{[K0]} Kuhlmann, F.-V.: {\it Quantifier elimination for
henselian fields relative to additive and multiplicative congruences},
Israel J. Math.\ {\bf 85} (1994), 277--306

\bibitem{[K1]} Kuhlmann, F.-V.: {\it Elementary properties of power
series fields over finite fields}, J. Symb. Logic, \textbf{66}, 771--791
(2001)

\bibitem{[K2]} Kuhlmann, F.-V.: {\it Value groups, residue fields
and bad places of rational function fields}, Trans.\ Amer.\ Math.\ Soc.\
{\bf 356} (2004), 4559--4600

\bibitem{[K2a]} Kuhlmann, F.-V.: {\it On places of algebraic function
fields in arbitrary characteristic}, Advanves in Math.\ {\bf 188}
(2004), 399--424

\bibitem{[K3]} Kuhlmann, F.--V.: {\it Additive polynomials and their
role in the model theory of valued fields}, Logic in Tehran, Lecture Notes in Logic \textbf{26}
160--203, Assoc.\ Symbol.\ Logic, La Jolla, CA, 2006

\bibitem{[K3a]} Kuhlmann, F.--V.: {\it Dense subfields of Henselian fields, and integer parts}, Logic in Tehran,
Lecture Notes in Logic {\bf 26} 204?-226, Assoc.\ Symbol.\ Logic, La Jolla, CA, 2006

\bibitem{[K4]} Kuhlmann, F.-V.: {\it The algebra and model theory of
tame valued fields}, to appear in: J.\ reine angew.\ Math. Preliminary
version published in: S\'eminaire de Structures Alg\'ebriques
Ordonn\'ees, \textbf{81}, Pr\'epublications Paris 7 (2009)

\bibitem{[K5]} Kuhlmann, F.-V.: {\it A classification of
Artin-Schreier defect extensions and a characterization of defectless
fields}, Illinois J. Math. {\bf 54} (2010), 397--448


\bibitem{[P]} Pop, F.: {\it Embedding problems over large fields},
Annals of Math.\ {\bf 144} (1996), 1--34.

\end{thebibliography}
\end{document}